\documentclass{amsart}
\usepackage{amssymb}
\usepackage[normalem]{ulem}
\usepackage{multirow}

\usepackage{url}
\usepackage{hyperref}
\newtheorem{theorem}{Theorem}

\newtheorem{conj}[theorem]{Conjecture}
\newtheorem{prop}[theorem]{Proposition}
\newtheorem{cor}[theorem]{Corollary}
\newtheorem{lemma}[theorem]{Lemma}

\mathchardef\myhyphen="2D

\newcommand{\FF}{\mathbb{F}}
\newcommand{\QQ}{\mathbb{Q}}
\newcommand{\TT}{\mathbb{T}}
\newcommand{\ZZ}{\mathbb{Z}}
\newcommand{\mm}{\mathfrak{m}}
\newcommand{\pp}{\mathfrak{p}}
\newcommand{\qq}{\mathfrak{q}}
\newcommand{\aaa}{\mathfrak{a}}

\newcommand{\OO}{\mathcal{O}}
\newcommand{\PP}{\mathcal{P}}

\newcommand{\barr}[1]{\overline{#1}}
\newcommand{\FFbar}{\barr{\FF}}
\newcommand{\FFtbar}{\FFbar_2}
\newcommand{\QQbar}{\barr{\QQ}}
\newcommand{\odd}{{\rm odd}}
\newcommand{\even}{{\rm even}}

\newcommand{\onto}{\twoheadrightarrow}
\newcommand{\into}{\hookrightarrow}

\DeclareMathOperator{\Gal}{Gal}

\DeclareMathOperator{\GL}{GL}
\DeclareMathOperator{\PGL}{PGL}
\DeclareMathOperator{\SL}{SL}
\DeclareMathOperator{\PSL}{PSL}
\DeclareMathOperator{\Ind}{Ind}
\DeclareMathOperator{\frob}{Frob}

\DeclareMathOperator{\Trace}{Tr}
\DeclareMathOperator{\cond}{cond}

\newcommand{\new}{{\rm new}}
\newcommand{\ord}{{\rm ord}}
\newenvironment{psmallmatrix}
	{\left( \begin{smallmatrix}}
	{\end{smallmatrix} \right)}
\DeclareMathOperator{\cl}{Cl}

\newcommand{\twosplit}{{\rm 2\myhyphen split}}
\newcommand{\CC}{\mathbb C}
\newcommand{\eps}{\varepsilon}
\newcommand{\tr}{{\rm tr}}

\newcommand{\KK}[1]{\QQ(\sqrt{#1})}
\newcommand{\KNp}{\KK{N}}
\newcommand{\KNm}{\KK{-N}}
\newcommand{\KNpm}{\KK{\pm N}}

\newcommand{\cmod}[1]{\ \mathrm{mod}\ #1}

\title[Mod-2 dihedral Galois representations]{Mod-2 dihedral Galois representations of prime conductor}

\author{Kiran S. Kedlaya}
\address{Univ. of California, San Diego, 9500 Gilman Drive \#0112, La Jolla, CA 92093 USA}
\email{kedlaya@ucsd.edu}
\urladdr{\url{http://kskedlaya.org/}}

\author{Anna Medvedovsky}
\address{Department of Mathematics and Statistics, Boston University, 111 Cummington Mall, Boston, MA 02215 USA}
\email{medved@gmail.com}

\thanks{The first author was supported by NSF (grant DMS-1501214) and UC San Diego (Warschawski Professorship). The second author was supported by an NSF postdoctoral research fellowship (grant DMS-1703834) and has gratefully enjoyed the hospitality of the Max Planck Institute for Mathematics during the writing of this paper.}

\begin{document}

\begin{abstract}
For all odd primes $N$ up to $500000$, we compute the action of the Hecke operator $T_2$ on the space $S_2(\Gamma_0(N), \QQ)$ and determine whether or not the reduction mod 2 (with respect to a suitable basis) has 0 and/or 1 as eigenvalues. We then partially
explain the results in terms of class field theory and modular mod-2 Galois representations.
As a byproduct, we obtain some nonexistence results on elliptic curves and modular forms with certain mod-2 reductions, extending prior results of Setzer, Hadano, and Kida.
\end{abstract}

\maketitle

\section{Introduction}

\subsection{Computations and theorems}

For $N$ a positive integer and $k$ a positive even integer, let
$S_k(\Gamma_0(N), \QQ)$ be the space of weight-$k$ rational cusp forms for the group $\Gamma_0(N)$,
equipped with the Hecke operators $T_p$ for all primes $p$ not dividing $N$.
For $N$ prime with $2 < N < 500000$, 
we computed the matrix of $T_2$ acting on some basis of $S_2(\Gamma_0(N), \QQ)$;
this was done using Cremona's implementation of modular symbols, as documented in \cite{cremona},
via the \texttt{eclib} package in \texttt{Sage} \cite{sage}. 
We then used the \texttt{m4ri} package in \texttt{Sage}, which implements the
``method of four Russians'' \cite[Chapter~9]{bard}, to compute the rank of the reductions of $T_2$ and $T_2 - 1$ mod 2. These computations took a few CPU-months; we did not make an accurate costing because our method is almost certainly not optimal (see below). 

From this data, we observed the following behavior of the mod-2 matrix of $T_2$.
\begin{itemize}
\item For $N \equiv 3 \cmod{8}$, the eigenvalue $0$ always occurs if $N > 3$.
\item For $N \equiv 1,3,5 \cmod{8}$, the eigenvalue 1 always occurs if $N > 163$.
\item For $N \equiv 1 \cmod{8}$, the eigenvalue $0$ occurs with probability $16.8\%$.
\item For $N \equiv 5 \cmod{8}$, the eigenvalue $0$ occurs with probability $42.2\%$.
\item For $N \equiv 7 \cmod{8}$, the eigenvalue $0$ occurs with probability $17.3\%$.
\item For $N \equiv 7 \cmod{8}$, the eigenvalue $1$ occurs with probability $47.9\%$.
\end{itemize}
These results can be partially explained (see section~\ref{conclusions}) by combining the Cohen-Lenstra heuristics \cite{cohen-lenstra} with a detailed count of the maximal ideals of the mod-2 Hecke algebra with residue field $\FF_2$. The bulk of the paper is devoted to making these counts (Theorems \ref{ec2} and \ref{mf2}) using class field theory plus  the theory of modular Galois representations. As a byproduct, we recover some nonexistence results of Setzer \cite{setzer}, Hadano \cite{hadano}, and Kida \cite{kida} for elliptic curves of conductor $N$ or $2N$ with $N$ prime,
derived using a totally different approach: a diophantine analysis of discriminants of Weierstrass equations due to Ogg \cite{ogg}.

For $N < 200000$, we also computed the multiplicities of 0 and 1 as generalized eigenvalues of the mod-2 reduction of the matrix of $T_2$. (These multiplicities are independent of the choice of basis.)
These are somewhat more complicated to analyze because the self-adjointness of $T_p$ with respect to the Petersson inner product does not guarantee diagonalizability mod $\ell$;
hence the computed multiplicity is an upper bound for the count of maximal ideals, and either both are zero or both are nonzero, but more work is needed to explain the full multiplicity.
See Conjecture \ref{mfmult} for a step in this direction; existing work on failure of multiplicity one in characteristic 2 (e.g., \cite{kilford-wiese}) suggests that even conjecturally, it may be difficult to formulate a more precise conjecture without allowing for some sporadic exceptions.

\subsection{Motivation: tabulation of rational eigenforms}
\label{subsec:motivation}

Although these results may be of independent interest, for context we indicate how they were motivated by some considerations around the tabulation of rational eigenforms.
Via the modularity theorem,  isogeny classes of elliptic curves of conductor $N$ 
correspond to rational newforms in $S_2(\Gamma_0(N), \QQ)$; finding rational eigenforms 
within $S_2(\Gamma_0(N), \QQ)$ is the rate-limiting step in Cremona's algorithm for tabulating rational elliptic curves of a given conductor, as documented in \cite{cremona} and  executed to date for $N \leq 400000$ \cite{lmfdb}.
(The table is also available in \texttt{PARI/GP} \cite{pari}, \texttt{Magma} \cite{magma}, and \texttt{Sage}.)

Within this step of Cremona's algorithm, the rate-limiting substep is the computation of the kernel of $T_p - a_p$
where $p$ is the smallest prime not dividing $N$ and $a_p$ runs over all integers with $\left| a_p \right| \leq 2\sqrt{p}$. Once this step is done, the resulting kernels are typically of much smaller dimension than the original space, so it is of negligible difficulty to diagonalize the restrictions of enough additional Hecke operators to isolate all one-dimensional joint eigenspaces. (The fact that this catches all rational eigenforms is a consequence of self-adjointness and strong multiplicity one.)

Recall that linear algebra over $\QQ$ is not generally performed using generic algorithms due to intermediate coefficient explosion; it is better to use a multimodular approach in which one does linear algebra over
$\FF_\ell$ for various small primes $\ell$ and reconstructs the final answer using the Chinese remainder theorem.
In Cremona's implementation of his algorithm, he uses only the single prime $\ell = 2^{30} - 35$; to date, this has provided enough information to identify the kernel of $T_p-a_p$.

The present work was motivated by a desire to understand the following question: to what extent (if any) can 
this algorithm be accelerated using linear algebra over $\FF_\ell$ for a single small $\ell$, such as $\ell=2$?
Of course, one does not expect the result of computing the kernel of $T_p - a_p$ mod $\ell$ to provide enough information to identify the kernel over $\QQ$. However, for $N$ large, the probability that $S_2(\Gamma_0(N), \QQ)$ admits \empty{any} rational newforms is relatively small: by analogy with the corresponding estimate for elliptic curves sorted by na\"\i ve height \cite{brumer} or Faltings height \cite{hortsch}, one expects that only $O(X^{5/6})$ of positive integers up to $X$ occur as levels of rational newforms. Consequently, there are likely to be many values of $N$ for which $T_p - a_p$ has no kernel at all over $\QQ$; if this remains true mod $\ell$, then finding this out would provide an early abort mechanism. A more sophisticated early abort strategy would be to
calculate not the rank of $T_p - a_p$, but rather
\begin{gather*}
\mbox{(contribution from level $N$ newforms)} \\
=
\mbox{(eigenvalue multiplicity of 0)}
- \sum_{d<N, d|N} \tau(N/d)
\mbox{(contribution from level $d$ newforms)}
\end{gather*}
where $\tau(n)$ is the number of divisors of $n$; an early abort occurs if this contribution does not increase under reduction modulo $\ell$.

The restriction to $N$ prime in this paper was made for several reasons; notably, a key role in the theoretical analysis is played by Eisenstein ideals, which are well understood for $N$ prime by the work of Mazur \cite{mazur} but remain largely mysterious for general $N$ (but still tractable for squarefree $N$, as in the work of Yoo \cite{yoo}). However, for $N$ prime there is no need to optimize Cremona's method: the method used by Bennett--Rechnitzer \cite{bennett-rechnitzer} to extend the tables of Stein--Watkins \cite{stein-watkins} is sufficient to compute (rigorously) a table of elliptic curves of all prime conductors up to $10^{10}$. Nonetheless, we hope that a thorough understanding of the present situation will provide a blueprint for extending the analysis; see below.

\subsection{Additional questions}
\label{subsec:future directions}

We conclude this introduction with discussion of further work to be done in this direction.
To begin with, our final analysis of the experimental data remains somewhat incomplete
because our analysis of mod-2 Galois representations focuses on the ones with dihedral image;
while representations with larger image are somewhat rarer, they do appear to make measurable contributions which we would like to see quantified.

In addition, one could repeat the analysis in alternate situations: one could treat nonprime $N$, work modulo another prime $\ell$, consider $T_p$ for another $p$,
and/or work in some higher weight $k$.
While all of these variants are of intrinsic interest, we would like to point out some developments in the computation of modular forms which  draw attention to some particular cases.
(Separately, the case of $N$ prime, $\ell = 2$, $p>2$, $k=2$ has arisen in the context of error-correcting codes \cite{randriam}.)

We first reconsider our choice of method to compute the Hecke actions on $S_k(\Gamma_0(N), \QQ)$.
The method of modular symbols is implemented in \texttt{Magma} \cite{magma} and \texttt{Sage} \cite{sage}, and in a specially optimized form for $k=2$ in Cremona's \texttt{eclib}. The approach used in \texttt{PARI/GP} \cite{pari} is based on trace formulas. However, for a large-scale tabulation of rational eigenforms, we believe the best approach is the method of \cite{birch} as extended by Hein--Tornar\'ia--Voight \cite{hein} (see also \cite{voight-agct}). Birch's original method is a variant of the Mestre--Oesterl\'e method of graphs \cite{mestre} in the case where $k=2$ and $N$ is prime; Birch (partially) described his method for $k=2$ and $N$ squarefree, in terms of reduction of definite quadratic forms,
while Hein--Tornar\'\i a--Voight generalize to higher weight by considering the action of $\mathrm{SO}(3)$ on nonstandard representations. Hein \cite{hein-github} has implemented the method in C++ for $k=2$ and $N$ squarefree; experimenting with this code reveals several computational benefits.
\begin{itemize}
\item
It is extremely efficient in practice.
\item
The matrix of $T_p$ is guaranteed\footnote{This is not true in Cremona's setup because projecting onto the minus part of the space of modular symbols could in principle introduce a denominator of 2; we have yet to observe this.} to be integral (but not symmetric) and optimally sparse, with at most $p+1$ nonzero entries per row.
\item
It separates eigenspaces for the Atkin-Lehner involutions, thus reducing the complexity of the resulting linear algebra.
\item
It removes some oldforms, thus again simplifying the linear algebra. For example,
if $N$ is squarefree with an odd number of prime factors, then no oldforms appear;
if $N$ is squarefree with an even number of prime factors, one gets an old subspace from the smallest
prime factor of $N$. For general $N$, one sees oldforms from levels which differ from $N$ by a square factor.
\end{itemize}

The early abort strategy of computing ranks modulo $\ell$ is potentially even more effective when using the Birch--Hein--Tornar\'\i a--Voight method, due to the separation of Atkin--Lehner eigenspaces. However, in order to realize this benefit one must probably take $\ell > 2$, as for $\ell = 2$ the two possible eigenvalues of an involution come together, so there is the chance of some problematic (for our purposes) interaction between the eigenspaces.
An analysis of the case $k=2$, $N$ prime, $\ell=3$ would be a natural variant of what we have done here.

Moreover, for $k>2$ the early abort strategy may be of even greater value, 
as rational newforms in $S_k(\Gamma_0(N), \QQ)$ correspond to Galois representations for which there is no systematic construction available. Indeed, there is some evidence that there are only finitely many such forms for $k > 4$ \cite{roberts}; extending previous exhaustive searches, particularly in the borderline case $k=4$, would be a natural next step.

\subsection{Acknowledgments} The authors thank Frank Calegari, Fred Diamond, Robert Pollack, Gabor Wiese, and Hwajong Yoo for helpful conversations.

\section{Elliptic curves and their 2-torsion}\label{ecsearchintro}

For $K$ a quadratic extension of $\QQ$, write $\OO_K$ for its ring of integers, $\cl(K)$ for its class group, $h(K)$ for its class number, and $H(K)$ for its Hilbert class field. 
Write $\cl(K, \aaa)$ for the ray class group of $K$ with conductor $\aaa$ and $h(K,\aaa)$ for the order of $\cl(K,\aaa)$.
Let $\pp(K)$ be a prime of $K$ above $(2)$, and write $\langle \pp(K) \rangle \subset \cl(K)$ for the subgroup that $\pp(K)$ generates. If $K$ is real, let $u(K)$ be a fundamental unit of $K$. 

For $E$ an elliptic curve, write $N_E$ for the conductor of $E$.
Let ${\bar\rho_{E, 2}: G_{\QQ, 2N_E} \to \GL_2(\FF_2)}$
be the mod-2 Galois representation associated to $E$; it factors through $G_{K_E}$ where $K_E := \QQ(E[2])$ has Galois group contained in $\GL_2(\FF_2) \cong S_3$.
By considering the subgroups of $S_3$ and their embeddings in $\GL_2(\FF_2)$, we see that exactly one of the following alternatives holds.
\begin{enumerate}
\item $E[2]$ is reducible as a Galois module, and $K_E$ is either $\QQ$ or a quadratic extension of $\QQ$ unramified away from $2N$.  In other words, $E$ has at least one rational $2$-torsion point.
\item $E[2]$ is irreducible over $\FF_2$ but becomes reducible over $\FF_4$, and $K_E$ is a cubic Galois extension of $\QQ$. In other words, $G_\QQ$ permutes the three non-identity points of $E[2]$ cyclically.\footnote{This happens, for example, for both isogeny classes of elliptic curves of conductor 196 (\url{http://www.lmfdb.org/EllipticCurve/Q/196/}) and isogeny classes {\tt a} and {\tt c} of conductor 324 (\url{http://www.lmfdb.org/EllipticCurve/Q/324/}).}
\item $E[2]$ is absolutely irreducible over $\FF_2$, and $K_E$ is an $S_3$-extension of $\QQ$. 
\end{enumerate}

\begin{prop} \label{P:mod 2 rep for squarefree conductor}
 If $N_E = 2^r M$ for some odd squarefree integer $M$ and some $r \geq 0$, then $E[2]$ is either reducible or absolutely irreducible.
\end{prop}

\begin{proof}
Suppose to the contrary that $K_E$ is cubic. Let $\ell$ be an odd prime dividing $N_E$. Since $\ell$ divides $N_E$ exactly once, $E$ has multiplicative reduction at $\ell$; hence the action
of $G_{\QQ_\ell}$ on the 2-adic Tate module of $E$ is reducible, and likewise for the action on $E[2]$.
 However, the (unique) order-$3$ subgroup of $\GL_2(\FF_2)$ is $\{ \begin{psmallmatrix} 1 & 0 \\ 0 & 1\end{psmallmatrix}, \begin{psmallmatrix} 0 & 1 \\ 1 & 1\end{psmallmatrix}, \begin{psmallmatrix} 1 &1 \\ 1 & 0\end{psmallmatrix}\}$, which acts irreducibly. Therefore the image of $G_{\QQ_\ell}$ is trivial in $\GL_2(\FF_2)$, and so $K_E$ is unramified at $\ell$. Since this is true for every odd $\ell$ dividing $N_E$, $K_E$ is ramified at most at $2$. But there are no cubic extensions of $\QQ$ unramified outside $2$: the maximal abelian extension unramified outside $2$ is $\QQ(\zeta_{2^\infty})$, whose Galois group is pro-$2$.
\end{proof}

In light of Proposition~\ref{P:mod 2 rep for squarefree conductor}, when $N_E$ is squarefree, we say that $E$ is \emph{reducible} if $E[2]$ is a reducible representation of $G_{\QQ}$ and \emph{$K$-dihedral}, or simply \emph{dihedral}, if $K_E$ is an $S_3$-extension containing a quadratic extension $K$ of $\QQ$.

Recall that $E$ is \emph{ordinary (at $2$)} if $a_2(E)$ is odd, and \emph{supersingular (at $2$)} otherwise. By theorems of Deligne and Fontaine (see Theorem \ref{delignefontaine}), $E$ is ordinary at $2$ if and only if $\bar\rho_{E, 2}|_{G_{\QQ_2}}$ is reducible. In particular, reducible elliptic curves are ordinary.

The following theorem will be proved in section~\ref{sec:proof of Theorem ec2}.

\begin{theorem}\label{ec2}
Let $N$ be an odd prime.
\begin{enumerate}
\item\label{dihedral} Every dihedral elliptic curve of conductor $N$ is either $\KNp$-dihedral or $\KNm$-dihedral.
\item \label{plusminusN}  {\bf Ordinary dihedral elliptic curves:}
For $K = \KNpm$, if $3 \nmid \displaystyle \frac{h(K)}{\# \langle \pp(K) \rangle}$, then there are no ordinary $K$-dihedral elliptic curves of conductor $N$.
\item\label{ss} {\bf Supersingular elliptic curves.}
\begin{enumerate}
\item \label{ss17} If $N \equiv 1, 7 \cmod{8}$, then there are no supersingular elliptic curves of conductor~$N$.
\item\label{ss3} If $N \equiv 3 \cmod{8}$, then every supersingular elliptic curve of conductor $N$ is $\KNm$-dihedral. 
\item \label{ss5} If $N \equiv 5 \cmod{8}$, then every supersingular elliptic curve of conductor $N$ is $\KNp$-dihedral. If $u(K) \not\equiv 1 \cmod{2 \OO_K}$, then there are no supersingular elliptic curves of conductor $N$.
\end{enumerate}
\item\label{eisenstein} {\bf Reducible elliptic curves:} If $N \not \equiv 1 \cmod{8}$, then there are no reducible elliptic curves of conductor $N$.
\end{enumerate}
\end{theorem}

For prime $N$ and $K = \KNpm$, the order of $\pp(K)$ in $\cl(K)$ divides $2$ unless $N \equiv 1 \cmod{8}$, so if $N \equiv 3, 5, 7 \cmod{8}$ then the condition $3 \nmid \frac{h(K)}{\# \langle \pp(K) \rangle}$ in \eqref{plusminusN} is equivalent to $3 \nmid h(K)$. Similarly, if $N \not\equiv 7 \cmod{8}$ and $K = \KNm$, then the condition $3 \nmid \frac{h(K)}{\# \langle \pp(K) \rangle}$  in \eqref{plusminusN} is equivalent to $3 \nmid h(K)$.

Theorem~\ref{ec2} includes a theorem of Setzer \cite[Theorem~1]{setzer}: if $N$ is a prime congruent to 1 or 7 mod 8
such that $3\nmid h(\KNpm)$, then every elliptic curve of conductor $N$ is reducible.
With similar methods, we also recover the following results of Hadano \cite[Theorem~II, Theorem~III]{hadano} and Kida \cite[Theorem~3.3]{kida}. (Kida's original statement requires $N-64$ to not be a square; for $N \neq 17$, this is equivalent to existence of a reducible elliptic curve of conductor $N$ \cite[Theorem~2]{setzer}.
 See also \cite[Theorem~I]{hadano}.)

\begin{theorem}[Hadano] \label{hadano}
Let $N$ be a prime such that $3 \nmid h(\KNpm), h(\KK{\pm 2N})$.
\begin{enumerate}
\item
If $N \equiv 1,7 \cmod{8}$, then every elliptic curve of conductor $2N$ is reducible.
\item
If $N\equiv 3,5 \cmod{8}$, there are no elliptic curves of conductor $2N$.
\end{enumerate}
\end{theorem}

\begin{theorem}[Kida] \label{kida}
Let $N$ be a prime such that none of
\[
h(\KNpm), \quad h(\KK{(-1)^{(N-1)/2} N}, 2)
\]
is divisible by $3$. Then every elliptic curve of conductor $N$ is reducible.
\end{theorem}

\section{Representation theory preliminaries}

To prepare for the proof of Theorem~\ref{ec2}, we make some representation-theoretic calculations.
Fix a prime $p$ and a field $\FF$ of characteristic $p$, let $G$ be any group, and let $\rho: G \to \GL_2(\FF)$ be a semisimple representation. Let $\rho(G) \subset \GL_2(\FF)$ and $\widetilde{\rho(G)} \subset \PGL_2(\FF)$ be the image and projective image of $\rho$, respectively.
Then exactly one of the following statements holds \cite[Propositions 15--16]{serre:ecs}.
\begin{enumerate}
\item {\bf Reducible case:} $\widetilde{\rho(G)}$ is a cyclic group $C_n$. In other words, $\rho$ is reducible (over $\bar\FF$), a sum of two characters $\chi \oplus \chi'$, and the order of $\chi/\chi'$ is $n$.
\item {\bf Dihedral case:} $\widetilde{\rho(G)}$ is a dihedral group $D_n$ of order $2n$ with $n \geq 2$.
In other words, $\rho$ is irreducible but there is an index-$2$ subgroup $H$ of $G$, determined uniquely if $n \geq 3$, so that $\left.\rho\right|_H$ splits as a sum of two characters.
\item {\bf Exceptional case:} $\widetilde{\rho(G)}$ is isomorphic to $A_4$, $S_4$, or $A_5$.
\item {\bf Big-image case:} $\widetilde{\rho(G)}$ contains $\PSL_2(\FF_q)$ for some $q \geq 5$, but ${\rho(G) \not= \SL_2(\FF_5)}$.\footnote{The restrictions are explained by exceptional isomorphisms for small primes: $\SL_2(\FF_2) \cong D_3$, $\PSL_2(\FF_3) \cong A_4$, $\PGL_2(\FF_3) \cong S_4$, $\PSL_2(\FF_4) = \PGL_2(\FF_4) \cong A_5$, and ${\PSL_2(\FF_5) \cong A_5}$.}
\end{enumerate}

Call $\rho$ \emph{reducible}, \emph{dihedral}, \emph{exceptional}, or \emph{big-image} accordingly.

\subsection{The dihedral case in detail}

\subsubsection{Inducing a character}\label{inducedrep} Let $H \subset G$ be a normal subgroup. Any character $\psi: H \to F^\times$ to a field $F$ may be twisted by any $g \in G$ to obtain a new character ${}^{g} \psi$, defined by ${}^{g} \psi(h) : = \psi(g^{-1} h g)$. Because $\psi$ factors though an abelian quotient of $H$, one can show that ${}^{g} \psi$ depends only on the class $\bar g$ of $g$ in $G/H$. We therefore write ${}^{\bar g} \psi$ for the twist of $\psi$ by $\bar g \in G/H$.

Now suppose that $H \subset G$ has index $2$ and take $\rho$ to be the induced representation $\Ind_H^G \psi : G \to \GL_2(F)$. Let $\eps_H$ be the (at most quadratic) character of $G$ that takes $H$ to 1 and $G-H$ to $-1$. Let $\bar g$ be the nontrivial element of $G/H$.
The following are well-known (e.g., see \cite[7.2.1]{serre:weight1}):
\begin{enumerate}
\item $\left.\rho\right|_H = \psi \oplus {}^{\bar g} \psi;$ 
\item\label{whenirred} $\rho$ is an irreducible representation of $G$ if and only if $\psi \neq {}^{\bar g} \psi$;
\item $\det \rho = \eps_{H} \cdot \psi ({\rm Ver}^G_H)$,
where ${\rm Ver}^G_H : G \to H^{\rm ab}$ is the \emph{Verlagerung} (transfer) homomorphism taking $x \in G$ to $ x \, g^{-1} x g$\footnote{One can show that $\psi ({\rm Ver}^G_H)$ takes $x \in H$ to $\psi \, {}^{\bar g} \psi(x)$ and takes $x \in G - H$ to $\psi(x^2)$.};
\item \label{whatisn} $\widetilde{\rho(G)} \cong D_n$, where $n$ is the order of ${}^{\bar g} \psi/\psi$ (assuming $\psi$ has finite order).
\end{enumerate}

\subsubsection{Dihedral representations} Conversely, suppose that $\rho: G \to \GL_2(F)$ is a dihedral representation
with $\widetilde{\rho(G)}  = D_n$. If $n \geq 3$, then  $D_n$ contains a unique index-$2$ subgroup isomorphic to $C_n$.\footnote{For $n=2$, there are three such subgroups. But $n$ is the order of a character to $\bar\FF_p^\times$ (see section \ref{inducedrep} \eqref{whatisn}) and hence prime to $p$; as we will later restrict to $p = 2$, we ignore $n = 2$ here.} Let $H \subset G$ be the inverse image of that cyclic subgroup under the map $G \to \GL_2(F) \to \PGL_2(F)$. Since $\widetilde{\rho(H)}$ is a cyclic group, $\left.\rho\right|_H$ is a reducible representation, a sum of two characters, each defined over an at-most-quadratic extension of $F$. Let $\psi: H \to \bar F^\times$ be one of these characters. Then Frobenius reciprocity and dimension considerations guarantee that the map $\Ind_{H}^G \psi \to \rho$ induced by $\psi \to \left.\rho\right|_H$ is an isomorphism.

\subsubsection{The image of a dihedral representation} Suppose further that $\rho$ is a faithful dihedral representation of $G$. With $H$, $\psi$, and ${}^{\bar g}\psi$ as above, we have the following:
\begin{lemma}
\begin{enumerate}
\item \label{kerker} $\ker \psi \cap \ker {}^{\bar g}\psi = 1$.
\item \label{Habelian} $H$ is an abelian subgroup of $G$.
\item If $\ker \psi \subset H$ is normal in $G$, then $\psi$ is faithful, so $H$ is cyclic.
\end{enumerate}
\end{lemma}

The proofs are straightforward but not completely standard, so we include them.

\begin{proof}
\begin{enumerate}
\item Indeed, $\left.\rho\right|_H = \psi \oplus {}^{\bar g}\psi$ and we have assumed that $\ker \rho$ is trivial.
\item The commutator of any two elements of $H$ is in both $\ker \psi$ and $\ker {}^{\bar g} \psi$; the claim follows from part \eqref{kerker}.
\item By part \eqref{Habelian}, $G/H$ acts on $H$ by conjugation, and $\ker {}^{\bar g} \psi$ is the image of $\ker \psi$ under the action of the nontrivial element. Now use \eqref{kerker}.
\end{enumerate}
\vspace{-10pt}
\end{proof}

Note that even if $\psi$ is faithful and $H$ is finite cyclic of order $n$ and the sequence $$1 \to H \to G \to G/H \to 1$$ splits (i.e., there is an order-$2$ element in $G - H$), we cannot conclude that $G$ is isomorphic to $D_n$: the dicyclic groups give a counterexample for every even $n$.

\subsubsection{Translating to Galois representations}

Let $\rho: G_{\QQ} \to \GL_2(F)$ be a finite-image dihedral representation such that $\left| \widetilde{\rho(G_{\QQ})} \right| \geq 6$. Let $K$ be the quadratic extension of $\QQ$ for which
$[\widetilde{\rho(G_{\QQ})}:\widetilde{\rho(G_K)}] = 2$, so that $\left.\rho\right|_{G_K}$ is reducible. Let $\psi: G_K \to F^\times$ be a character appearing in $\left.\rho\right|_{G_K}$ and let $L_\psi$ be the fixed field of $\ker \psi$. If $L_{\psi}/\QQ$ is Galois, then $L_\psi = \ker \rho$.  Otherwise, writing $\Gal(K/\QQ)= \{1, \sigma\}$, we obtain the twist ${}^\sigma \psi$; its fixed field $L_{{}^\sigma \psi}$ is the image ${\tilde \sigma(L_{\psi})} \subset \bar\QQ$ for any lift $\tilde \sigma$ of $\sigma$ to $G_\QQ$; and $\ker \rho =: M$ is the compositum $L_\psi L_{{}^\sigma \psi}$ (inside $\bar\QQ$). In particular, it is clear that $M$ is an abelian extension of $K$.

\subsubsection{Artin conductor formulas}
We will also make use of the following formula (see, for example, \cite[Corollary 1]{taguchi}) for the Artin conductor of $\Ind_K^\QQ \psi$ in terms of the Artin conductor of $\psi$ :
\begin{equation}\label{artincond}
\cond(\Ind_K^\QQ \psi) = \left|\Delta_K\right| \:{\mathcal N}^{K}_{\QQ}\big(\!\cond\psi \big),
\end{equation}
where ${\mathcal N}^K_{\QQ}$ is the field norm and $\Delta_{K}$ is the discriminant of $K$.

If $F$ is a finite extension of $\FF_p$ or a $p$-adic field, we will denote the \emph{tame} or prime-to-$p$ Artin conductor by $\cond^{(p)}$.
The analogous formula holds:
\begin{equation}\label{tameartincond}
\cond^{(p)}(\Ind_K^\QQ \psi) = \left|\Delta^{(p)}_K\right| \:{\mathcal N}^{K}_{\QQ}\big(\!\cond^{(p)}\!\chi \big).
\end{equation}
Here $\Delta^{(p)}_{K}$ is the prime-to-$p$ part of the discriminant of $K$.

\subsection{Mod-$2$ dihedral Galois representations}
From now on, we work with ${F=\FF},$ a finite extension of $\FF_2$. Suppose that $\rho = \Ind_K^\QQ \psi: G_{\QQ} \to \GL_2(\FF)$ is a $K$-dihedral representation for some quadratic $K$ over $\QQ$ and ray class (i.e., Hecke) character $\psi: G_K \to \FF^\times$. 

\subsubsection{Implications of $\det \rho = 1$}

Again, let $L_\psi$ be the fixed field of $\ker \psi$.

\begin{lemma}\label{det1}
If $\det \rho = 1$, then $L_\psi$ is Galois over $\QQ$.
\end{lemma}

\begin{proof}
If $\det \rho = 1$, then considering $\det \rho$ on the subgroup $G_K$, we see that ${{}^\sigma \psi=\psi^{-1}}$. Therefore $L_\psi$ is also the fixed field of $\ker {}^\sigma \psi$, which means that $L_\psi/\QQ$ is Galois and $L_\psi$ is the fixed field of $\ker \rho$.
\end{proof}

\subsubsection{The conductor of $\psi$}
Let $\aaa$ be the conductor of $\psi$. Since we work in characteristic $2$, we are only interested in odd-order $\psi$ here; we thus ignore consideration of any real places of $K$ and view $\aaa$ as an integral ideal of $K$. We have a standard exact sequence relating the class group $\cl(K)$ to the ray class group $\cl(K, \aaa)$:
\begin{equation} \label{rayclasseq}
\OO_K^\times \to \big(\OO_K/\aaa \big)^\times \to \cl(K, \aaa) \to \cl(K) \to 1
\end{equation}

\begin{lemma}\label{onlyeven}
If $\aaa = \qq^n$ is a power of a prime of $\OO_K$ lying over a prime $q$ of $\ZZ$, then
\begin{equation*}
[\cl(K, \aaa): \cl(K)] \mbox{ divides }
\begin{cases}
(q-1) q^k \mbox{ for some $k \geq 0$} & \mbox{ if $(q)$ splits or ramifies in $K$,} \\
(q^2 - 1) q^k  \mbox{ for some $k \geq 0$}& \mbox{ if $(q)$ is inert in $K$.}
\end{cases}
\end{equation*}
\end{lemma}
\begin{proof}
Immediate from sequence \eqref{rayclasseq} in light of the exact sequence
\begin{equation}\label{pro2}1 \to 1 + \qq^n \OO_K \to 1 + \qq \OO_K \to (\OO_K/\qq^n)^\times \onto (\OO_K/\qq)^\times \to 1,
\end{equation} combined with the fact that $1 + \qq \OO_K$ is pro-$q$.
\end{proof}

\begin{cor}\label{2unram}
\begin{enumerate}
\item If $2$ ramifies or splits in $K$, then any Hecke character ${\psi: G_K \to \FF^\times}$ of modulus $2^n \OO_K$ has trivial conductor and hence factors through $\cl(K)$.
\item If $2$ is inert in $K$, then any Hecke character $\psi: G_K \to \FF^\times$ of modulus $2^n \OO_K$ has conductor dividing $2 \OO_K$ and hence factors through $\cl\!\big(K, (2)\big)$.
\end{enumerate}
\end{cor}

\begin{proof}
\begin{enumerate}
\item If $2$ ramifies in $K$, then this follows immediately from Lemma \ref{onlyeven}, since ${(q-1)q^n}$ is a power of $2$. If $2$ splits as $2\OO_K = \pp \pp'$, then argue as in Lemma \ref{onlyeven}, noting that $\big(\OO_K/(2 \OO_K)^n\big)^\times = (\OO_K/\pp^n)^\times \times \left(\OO_K/{\pp'}^n\right)^\times$ by the Chinese remainder theorem.
\item From the proof of Lemma \ref{onlyeven} and sequence \eqref{pro2}, it's clear that the only odd contribution to $[\cl\!\big(K, (2)^n\big): \cl(K)]$ comes at $n = 1$. 
\end{enumerate}
\vspace{-12pt}
\end{proof}

\subsubsection{The local behavior of $\rho$}\label{locallyred}
Fixing an embedding $\iota: G_{\QQ_2} \into G_\QQ$, we can consider the restriction $\rho_{2}$ of $\rho$ to $G_{\QQ_2}$.
Let $\pp$ be the prime of $\OO_K$ above $2$ corresponding to $\iota$, and let $\psi_2$ be the restriction of $\psi$ to $G_{K_\pp}$. Then $\rho_2$ is reducible if and only if either
\begin{enumerate}
\item $2$ splits in $K$, or
\item $2$ is inert or ramified in $K$ and ${}^{\sigma} \psi_2  = \psi_2$.\\ (Note that $\sigma$ is in the decomposition group at $\pp$ in this case.)
\end{enumerate}

\subsection{Mod-$2$ dihedral Galois representations of prime conductor}\label{2scenarios}
Retaining the notation ($\FF$, $\rho$, $K$, $\psi$) from the previous subsection, we now additionally suppose that $N$ is an odd prime and $\rho$ has (tame Artin) conductor $N$.
The induced tame conductor formula \eqref{tameartincond} guarantees that either
\begin{equation*}
\Delta^{(2)}_K = (1), \,\mathcal N^K_\QQ(\cond^{(2)}\! \psi) = (N) \quad
\mbox{or} \quad
\Delta^{(2)}_K = (N), \,\mathcal N^K_\QQ(\cond^{(2)}\! \psi) = (1).
\end{equation*}
We analyze each scenario in turn.

\subsubsection{First scenario: $\Delta^{(2)}_K = (1) \mbox{ and } \mathcal N^K_\QQ(\cond^{(2)} \psi) = (N)$}

Here, $K = \QQ(i)$ or $\KK{\pm2}$, and  $N$ splits in $K$ as $(N) = \qq \qq'$ with $\cond^{(2)}\psi = \qq$. Hence $\psi$ is a ray class character of conductor $\qq \aaa$ for some ideal $\aaa$ of $K$ divisible only by primes above~$2$. 

\begin{lemma}\label{baddet}
In this scenario, $\det \rho: G_\QQ \to \FF^\times$ is a nontrivial character.
\end{lemma}

\begin{proof}
Since $\cond\psi$ is not Galois-invariant, $L_\psi$ is not Galois over $\QQ$. Lemma \ref{det1} then implies the desired conclusion.
\end{proof}

\subsubsection{Second scenario: $\Delta^{(2)}_K = (N) \mbox{ and } \mathcal N^K_\QQ(\cond^{(2)} \psi) = (1)$}\label{scenario2}

Here, $K = \KNpm$ or $\KK{\pm 2N}$ and $\psi$ is a ray class character of conductor dividing $(2 \OO_K)^n$.

\begin{cor}\label{factorthruH}
In this scenario, $\psi$ factors through $\cl(K)$ unless
\begin{itemize}
\item $N \equiv 5 \cmod{8}$ and $K = \KNp$ or
\item $N \equiv 3 \cmod{8}$ and $K = \KNm$,
\end{itemize}
in which cases $\psi$ factors through $\cl\!\big(K, (2)\big).$
\end{cor}

\begin{proof} Combine Corollary \ref{2unram} with the ramification of $2$ in $\KNpm$: see Table~\ref{ordersparity}.
\end{proof}

\begin{table}[ht]
\caption{Class number parity and splitting of $2$ in $\KNpm$ for $N$ prime.}
\label{ordersparity}.

\begin{tabular}{c||c|c|c||c|c|c}
\multirow{2}{*}{$N$ mod $8$}
& \multicolumn{3}{c||} {$K = \KNp$}  & \multicolumn{3}{c} {$K = \KNm$} \\
\cline{2-7}
& $(2)$ in $K$ & $h(K)$ & $\#\langle \pp(K) \rangle$ &
$(2)$ in $K$ & $h(K)$ & $\#\langle \pp(K) \rangle$\\
\hline\hline
1 &
splits & odd & varies &
ramifies & even $> 4$ & 2 \\
\hline
3 &
ramifies & odd & 1 &
inert & odd & 1\\
\hline
5 &
inert & odd & 1 &
ramifies & $2\cdot$odd & 2\\
\hline
7 &
ramifies & odd & 1 &
splits & odd & varies
\end{tabular}

\end{table}

\subsection{Mod-$2$ modular Galois representations of weight 2}

We now suppose that $N$ is an odd integer (not necessarily prime) and $f \in S_2(\Gamma_0(N), \bar\ZZ_2)$ is a normalized weight-$2$ Hecke eigenform of level $N$. By a theorem of Breuil--Conrad--Diamond--Taylor \cite{bcdt}, such $f$ with coefficients in $\QQ$ correspond precisely to isogeny classes of elliptic curves $E$ of conductor $N$, with the $\ell^{\rm th}$ Fourier coefficient satisfying $a_\ell(f) = \ell + 1 - \#E(\FF_\ell)$ for all primes $\ell \nmid 2N$. As for elliptic curves, the form $f$ is \emph{ordinary} or \emph{supersingular} according to whether $a_2(f)$ is a unit in $\bar \ZZ_2$. Reducing any $G_\QQ$-stable lattice of the Galois representation associated by Eichler and Shimura to $f$, we obtain a mod-$2$ representation $\rho_f: G_{\QQ} \to \SL_2(\bar\FF_2)$ which for prime $\ell \nmid 2N$ is unramified at $\ell$ and satisfies $\Trace \rho_f(\frob_\ell) = \bar a_\ell(f)$, where $\bar a_\ell(f) \in \bar \FF_2$ is the mod-$2$ reduction of $a_\ell(f)$. If $f$ corresponds to an elliptic curve $E$ (up to isogeny,) then $\rho_f$ is the representation $\rho_{E, 2}$ (up to semisimplification) discussed in section \ref{ecsearchintro}. 

Fixing a prime of $\bar \QQ$ above $2$, we consider the corresponding decomposition group of $G_\QQ$, which one can identify with the absolute Galois group $G_{\QQ_2} = \Gal(\bar\QQ_2 / \QQ_2)$ of the local field $\QQ_2$. The following theorem relates the shape of the local representation $\rho_{f, 2} := \left. \rho_f \right|_{G_{\QQ_2}}$ to the invertibility of $a_2(f)$. In the statement and the proof, $\QQ_{p^2}$ refers to the unique unramified degree-$2$ extension of $\QQ_p$.

\begin{theorem}[Deligne, Fontaine, Edixhoven, Serre]\label{delignefontaine}
One of the following holds.
\begin{enumerate}
\item \label{dford} ${\rho}_{f, 2}$ is reducible, in which case $f$ is ordinary, and
$${\rho}_{f, 2} \sim \begin{pmatrix} \lambda^{-1} & \ast \\ 0 & \lambda \end{pmatrix},$$
where $\lambda: G_{\QQ_2} \to \FFtbar^\times$ is the unramified character sending $\frob_2$ to $\bar a_2(f)$.\\ Moreover $\rho_{f, 2}$ is at most peu wildly ramified in the sense of Serre.\footnote{An extension $M/\QQ_p$ is \emph{at most peu wildly ramified} if $M = M^{\tr}(\alpha_1^{1/p}, \ldots, \alpha_d^{1/p})$, where $M^\tr/\QQ_p$ is the at most tamely ramified subextension of $M$, and the $\alpha_i$ can be taken to be units in $M^\tr$. If $M$ is still an elementary $p$-extension of $M^\tr$ but at least one of the $\alpha_i$ must be a nonunit, then $M$ is \emph{tr\`es wildly ramified}. See \cite[2.4.ii]{serre:conj}. A representation of $\Gal(\bar \QQ_p/\QQ_p)$ as usual inherits the ramification properties of the fixed field of its kernel.}

\item \label{dfss} ${\rho_{f, 2}}$ is irreducible, in which case $f$ is supersingular. In this case, $\rho_{f, 2}$ is the induction of a character of $G_{\QQ_{4}}$ (the second fundamental character) and is therefore at most tamely ramified.
\end{enumerate}
\end{theorem}

\begin{proof}
Write $p$ in place of $2$ to avoid confusion with weight $2$. For the shape of $\rho_{f, p}$, see Edixhoven \cite[Theorems 2.5, 2.6]{edix}. In the ordinary case, since $f$ has level prime to $p$ and weight $2$, $\rho_{f, p}$
is \emph{finite at $p$}: it arises from a finite flat group scheme over $\bar\ZZ_p$ (the  $p$-torsion of a certain abelian variety of $\mathrm{GL}_2$-type), forcing $\rho_{f, p}$ to be at most peu wildly ramified \cite[Proposition 8.2]{edix}. In the supersingular case, $\rho_{f, 2}$ is at most tamely ramified by \cite[Proposition 4]{serre:ecs}; for the description of $\rho_{f, p}$ as the induction of the second fundamental character of $G_{\QQ_{p^2}}$, see \cite[\S 2.2]{serre:conj}.
\end{proof}

\section{Mod-2 dihedral representations appearing in weight 2}

Before proving Theorem \ref{ec2}, we state an analogous theorem for cuspforms of weight 2: see Theorem \ref{mf2} below. As many of the arguments are identical, the two theorems will be proved together in section \ref{sec:proof of Theorem ec2}. 

For $N$ an odd squarefree positive integer, we study the distribution of generalized $T_2$-eigenvalues on $S_2(\Gamma_0(N), \bar \FF_2)^\new$. Write $m(N)$ for the dimension of this space. For $\alpha \in \bar \FF_2$, write $m(N, \alpha)$ for the dimension of the generalized kernel of $T_2 - \alpha$ on this space (i.e., the dimension of the generalized eigenspace corresponding to $T_2$-eigenvalue $\alpha$). Let $m_{\ord}(N) := m(N) - m(N, 0)$, the dimension of the \emph{ordinary} subspace. Our aim will be to give lower bounds on $m_\ord(N)$, $m(N, 1)$, and $m(N, 0)$ by enumerating dihedral forms with multiplicities. Note that, for squarefree $N$, forms defined over $\FF_2$ will be either dihedral or reducible (that is, the analog of Proposition \ref{P:mod 2 rep for squarefree conductor} holds).

To this end, write $S_2(N) := S_2(\Gamma_0(N), \bar \FF_2)^\new$ and let $\TT_2(N) := \TT_2(N, \bar\FF_2)^\new$ be the shallow Hecke algebra acting on $S_2(N)$. In other words, $\TT_2(N)$ is the (commutative) $\bar\FF_2$-algebra generated inside ${\rm End}_{\bar \FF_2} \big(S_2(N) \big)$ by the action of all the Hecke operators $T_n$ with $n$ prime to $2N$. Then $\TT_2(N)$ is a semilocal artinian ring whose maximal ideals $\mm$ correspond to mod-$2$ Hecke eigensystems appearing in $S_2(N)$. For $\ell$ prime to $2N$, let $a_\ell(\mm) \in \bar\FF_2$ be the $T_\ell$-eigenvalue corresponding to $\mm$; note that $\mm$ is generated by the $T_\ell - a_\ell(\mm)$ for $\ell \nmid 2N$. By Serre reciprocity (a/k/a Serre's conjecture \cite{serre:conj}, now known by work of Khare--Wintenberger \cite{kw1, kw2}, Kisin \cite{kisin},
and Dieulefait \cite{dieulefait}), the maximal ideals $\mm$ also correspond to semisimple Galois representations $\rho_\mm: G_{\QQ, 2N} \to \SL_2(\bar\FF_2)$ that are at most peu wildly ramified at $2$. 
The correspondence is codified by the Eichler-Shimura relation $a_\ell(\mm) = \tr \rho_\mm(\frob_\ell)$. Theorem \ref{delignefontaine} implies that given $\mm$, one can determine whether $a_2(\mm)$ is $0$ or $1$; otherwise $a_2(\mm)$ is only defined up to inverse.\footnote{Note that $a_2(\mm)$ is not in general the trace of a Frobenius element at $2$ of the $\rho_\mm$ corresponding to $\mm$ (indeed, $\rho_\mm$ may be ramified at $2$). Therefore $a_2(\mm)$ is not a priori determined by~$\mm$. In fact, $a_2(\mm)$ may not even be defined over the field of definition of $\rho_\mm$. This happens, for example, in level 257 for the $\KK{257}$-dihedral Galois orbit of forms.}

We decompose $\TT_2(N)$ as a product of localizations at its maximal ideals,
and correspondingly decompose of $S_2(N)$ into generalized $\mm$-eigenspaces $S_2(N)_\mm$:
$$\TT_2(N) = \prod_{\mm} \TT_2(N)_\mm, \qquad S_2(N) = \bigoplus_{\mm} S_2(N)_\mm.$$
Note that if $\mm \subset \TT_2(N)$ is a maximal ideal, then the eigenspace $S_2(N)[\mm]$ is nonzero, so that the dimension of the generalized eigenspace $S_2(N)_\mm$ is at least $1$.

We say that a maximal ideal $\mm$ of $\TT_2(N)$ is \emph{reducible}, \emph{dihedral}, \emph{exceptional}, or \emph{big-image} if $\rho_\mm$ has the corresponding property. Similarly, we say that $\mm$ is \emph{supersingular} or \emph{ordinary} if $\rho_\mm$ is so at $2$.

We determine the fields $K$ for which there exist $K$-dihedral $\mm$ occurring in $\TT_2(N)$ for $N$ prime and how many such $\mm$ there are (Theorem \ref{mf2} below). In section~\ref{sec:mult}, we study the multiplicity of $S_2(N)_\mm$ in each case (Conjecture \ref{mfmult} and Proposition \ref{mfmultproof}).

\begin{theorem}\label{mf2}
Let $N$ be an odd prime, and $\mm \subset \TT_2(N)$ a maximal ideal.
\begin{enumerate}
\item If $\mm$ is dihedral, then it is either $\KNp$-dihedral or $\KNm$-dihedral.
\item \label{countord} {\bf Ordinary dihedrals:} For $K = \KK{\pm N}$, there are exactly $\displaystyle \frac{h(K)^{\odd} - 1}{2}$ ordinary $K$-dihedral maximal ideals in $\TT_2(N)$.
Of these,
$\displaystyle \frac{h(K)^{\odd, \twosplit} - 1}{2}$
have $a_2(\mm) = 1$.
\vspace{1pt}
\item {\bf Supersingular dihedrals.}
\begin{enumerate}
\item If $\mm$ is supersingular $K$-dihedral, then either $N \equiv 3 \cmod{8}$ and ${K = \KNm}$, or $N \equiv 5 \cmod{8}$ and $K = \KNp$.
\item\label{countss3} Let $N \equiv 3 \cmod{8}$ and $K = \KNm$. If $N > 3$, then there are exactly $h(K)$ supersingular maximal ideals of $\TT_2(N)$.
\item \label{countss5} Let $N \equiv 5 \cmod{8}$ and $K = \KNp$. If $u(K) \equiv 1 \cmod{2\OO_K}$, then there are $h(K)$ supersingular maximal ideals of $\TT_2(N)$; otherwise, there are none.
\end{enumerate}
\item \label{1reduce} {\bf Reducibles:} If $N \equiv 1 \cmod{8}$, then there is one reducible maximal ideal of $\TT_2(N)$, generated by $T_\ell$ for every prime $\ell \nmid 2N$; otherwise, there are none.
\end{enumerate}
\end{theorem}

Note that $h(\KNp)$ is always odd, and $h(\KNm)$ is even only for ${N \equiv 1 \cmod 4}$. Note also that a prime $\pp$ above $2$ of $K = \KK{\pm N}$ has order $1$ or $2$ in the class group unless $N \equiv \eps \cmod{8}$ and $K = \KK{\eps N}$ for $\eps = \pm 1$, so the $\twosplit$ condition is vacuous outside those two cases.

\section{Proofs of theorems}
\label{sec:proof of Theorem ec2}

We prove the various parts of Theorems \ref{ec2} and \ref{mf2} in parallel. We then adapt the ideas to recover the theorems of Hadano (Theorem~\ref{hadano}) and Kida (Theorem~\ref{kida}).

\subsection{Proof of parts \eqref{dihedral}} 
Suppose that $f \in S_2(N)$ is a $K$-dihedral modular form for some quadratic extension $K$ of $\QQ$ (corresponding to an elliptic curve for Theorem~\ref{ec2} or to a maximal ideal of the Hecke algebra for Theorem \ref{mf2}). Since $\rho_f$ factors through an extension of $\QQ$ unramified outside of $2$ and $N$, $K$ must be one of the following:
$$\KNp, \KNm, \KK{-1}, \KK{2}, \KK{-2}, \KK{2N}, \KK{-2N}.$$
If $K = \KK{\pm 2N}$, then $K$ is tr\`es wildly ramified at $2$ \cite[2.6, Exemple]{serre:conj}, so no modular forms of weight $2$ (and in particular no elliptic curves) can be $K$-dihedral (Theorem \ref{delignefontaine}).  If $K = \KK{-1}$, $\KK{-2}$, or $\KK{2}$, then we are in the first scenario of subsection \ref{2scenarios}, and Lemma \ref{baddet} guarantees that a $K$-dihedral representation cannot come from a $\Gamma_0(N)$-modular form. Thus $K = \KNpm$, as claimed.

\subsection{Proof of parts \eqref{plusminusN}} \label{sec:proofpart2} Suppose $K = \KNpm$ and $f \in S_2(N)$ is a $K$-dihedral ordinary form, with $\rho = \rho_f = \Ind_K^\QQ \psi$ for some character $\psi$ of $G_K$ ramified only at primes above $2$ (section \ref{scenario2}). Write $H= H(K)$ and $\pp = \pp(K)$. Let $L$ be the fixed field of $\ker \psi$. Since $\det \rho = 1$, by Lemma \ref{det1} the extension $L/\QQ$ is Galois. Choose a prime $\PP$ of $L$ above $\pp$, and write $\psi_2$ for the restriction of $\psi$ to $\Gal(L_\PP/K_\pp)$. 

We first show that $\psi$ is in fact unramified at $2$, and hence will factor through $H^\odd$, the maximal odd-degree subextension of $H$. 
By Corollary \ref{factorthruH} and Table~\ref{ordersparity}, $\psi$ is unramified in all cases except possibly when
$2$ is inert in $K$. In that case,  $\rho_{f,2}= \Ind_{K_\pp}^{\QQ_2} \psi_2$, so by \ref{inducedrep} \eqref{whenirred} we know that $\psi_2 = {}^{\sigma_2} \! \psi_2$ for $\sigma_2$ a  generator of $\Gal(K_\pp/\QQ_2)$. In this case, Theorem \ref{delignefontaine} \eqref{dford} tells us that $\psi_2$ is unramified above $2$, as then is $\psi$. In fact, the determinant condition further forces ${}^{\sigma_2} \! \psi_2 = \psi_2^{-1}$, which implies $\psi_2 = 1$ because we are in characteristic $2$. 

Next, from Theorem~\ref{delignefontaine}~\eqref{dford}, the condition $a_2(f) = 1$ is equivalent to the condition $\psi_2 = 1$, which exactly means that $\psi$ factors through $H^{\odd, \twosplit}$, the maximal odd subextension of $H$ over $K$ in which $2$ splits completely.

To complete the proof of Theorem \ref{ec2} \eqref{plusminusN}, we observe that ${[H^{\odd, \twosplit}: K] = \frac{h(K)^\odd}{\#\langle \pp \rangle}}$. If $\rho$ comes from a $K$-dihedral elliptic curve, then it has image $D_3$ so that $\psi$ must have order $3$. So a $K$-dihedral elliptic curve of conductor $N$ is only possible if $3$ divides $\frac{h(K)^\odd}{\#\langle \pp \rangle}$, or equivalently $\frac{h(K)}{\#\langle \pp \rangle}$.

To complete the proof of Theorem \ref{mf2} \eqref{plusminusN}, we recall that in general, $\Ind_K^\QQ \psi = \Ind_K^\QQ \psi'$ if and only if $\psi = \psi'$ or ${}^\sigma \psi = \psi'$ for $\sigma$ a generator of $\Gal(K/\QQ)$. In our unit-determinant case, ${}^\sigma \psi = \psi^{-1}$. Therefore there are $\frac{h(K)^{\odd} - 1}{2}$ distinct ordinary $K$-dihedral $\rho$, as claimed. The $a_2 = 1$ condition works similarly.

\subsection{Proof of parts \eqref{ss}}
\label{subsec:ss}

Suppose that $K = \KNpm$ and $f \in S_2(N)$ is a $K$-dihedral form with $\rho = \rho_f= \Ind_K^\QQ \psi$ for some character $\psi$ of $G_K$ ramified only at primes above $2$. Maintain the notation $H$, $\pp$, $\sigma$, $\rho_2$ as above. As in the second paragraph of section \ref{sec:proofpart2}, $\psi$ does not factor through $H$ (or else $\rho_2$ would be reducible, contradicting Theorem \ref{delignefontaine}).  
Therefore $\psi$ must be a character of $\cl(K,\aaa)$ for some ideal $\aaa$ of $K$ divisible only by primes above $2$. By Corollary \ref{2unram}, $\aaa = (2)$ and either $N \equiv 3 \cmod{8}$ and $K = \KNm$, or $N \equiv 5 \cmod{8}$ and $K = \KNp$.

Now suppose we are in one of these two cases. Since ${}^\sigma \psi = \psi^{-1}$, the character ${}^\sigma \psi$ will also factor through $H\!\big(K, (2)\big)$ and not through $H$. This gives exactly $\frac{h(K, (2)) - h(K)}{2}$ representations, and hence maximal ideals of $\TT_2(N)$. 

The formulations in part \eqref{countss3} of Theorem \ref{mf2} and part \eqref{countss5} of both theorems come from analyzing the sequence~\eqref{rayclasseq} from the proof of Lemma~\ref{onlyeven}. For $N$ congruent to $3$ modulo $8$, we have $K = \QQ(\sqrt{-N})$, so that $$\OO_K = \begin{cases} \{\pm 1 \} & \mbox{if $N > 3$} \\ \{\pm 1, \pm \omega, \pm \omega^2\} & \mbox{if $N = 3$} \end{cases}$$ for $\omega$ a cube root of unity in $\QQ(\sqrt{-3})$. Since $(2)$ is inert in $K$, we have $\OO_K/(2) = \FF_4$.
Therefore, for $N > 3$ (still congruent to $3$ modulo $8$), sequence \eqref{rayclasseq} becomes
$$\{\pm 1\} \to \FF_4^\times \to H(K, (2)) \to H(K) \to 1,$$ so that $h(K, (2)) = 3 h(K)$. For $N = 3$, on the other hand, the global units exactly cancel out the mod-$(2)$ units, so that $h(K, (2)) = h(K)$. For $N$ congruent to $5$ modulo $8$, we still have $\OO_K/(2) = \FF_4$, but this time $\OO_K = \{\pm 1\} \times u^\ZZ$ for some fundamental unit $u = u(K),$ and therefore we similarly have the two cases $$h(K, (2)) = \begin{cases} 3 h(K) & \mbox{if $u$ maps to $1$ in $\big(\OO_K/(2)\big){}^\times$} \\ h(K) & \mbox{otherwise.} \end{cases}$$

\subsection{Proof of parts \eqref{eisenstein}}
If $N \not \equiv 1 \cmod 8$, then $2$ is not an Eisenstein prime for $N$ (see Mazur \cite{mazur} or Mazur--Serre \cite{mazurserre}), so there are no cuspforms in $S_2(N, \bar\ZZ)$ congruent to the Eisenstein series $E_{2, N}$ modulo $2$, which carries the unique reducible maximal ideal in squarefree level. In particular, there are no rational newforms whose associated mod-$2$ Galois representation is reducible.

This completes the proof of Theorem \ref{ec2} and Theorem \ref{mf2}.

\subsection{Proof of Theorem~\ref{kida}}

By Theorem~\ref{mf2} \eqref{plusminusN}, the condition $3 \nmid h(\KNpm)$
rules out the existence of an ordinary elliptic curve of conductor $N$. 
For a supersingular elliptic curve,
with notation as in the proof of Theorem~\ref{mf2} \eqref{ss}, 
$K = \KK{(-1)^{(N-1)/2} N}$ and $\psi$ is a nontrivial order-3 character of $H(K, (2))$; this is ruled out by assuming that $3 \nmid h(K, (2))$.
This completes the proof of Theorem~\ref{kida}.

\subsection{Proof of Theorem~\ref{hadano}}
\label{subsec:hadano}

We now change notation to address Theorem~\ref{hadano}.
Let $N$ be a prime such that $3 \nmid h(K)$ for $K = \KNpm, \KK{\pm 2N}$,
and let $E$ be an elliptic curve of conductor $2N$. Let $f \in S_2(2N)$ be the corresponding modular form and let $\mm \subseteq \TT_2(2N)$ be the corresponding maximal ideal. 
Since $E$ has multiplicative reduction at 2, $f$ is ordinary and the conclusion of Theorem~\ref{delignefontaine} \eqref{dford} holds. By Proposition~\ref{P:mod 2 rep for squarefree conductor},
$\mm$ is either reducible or ordinary dihedral.

In the reducible case,
$\mm$ is an Eisenstein ideal; by the proof of \cite[Theorem~6.1]{yoo},
the difference of the cusps of $X_0(2N)$ corresponding to $1, 1/2 \in \mathbb{P}^1(\QQ)$ must have even order in the Jacobian. By \cite[Theorem~1.3]{yoo} this order is the numerator of $(N^2-1)/8$, forcing $N \equiv 1, 7 \cmod{8}$.

In the ordinary dihedral case, by Lemma~\ref{baddet} we must be in the second scenario of subsection \ref{2scenarios};
that is, that is, $\rho_f = \Ind_K^\QQ \psi$ where
$K$ is one of $\KK{\pm N}, \KK{\pm 2N}$ and $\psi$ is an order-3 character of $G_K$ ramified only at primes above $2$. 
As in subsection~\ref{sec:proofpart2}, we see that $\psi$ is also unramified at $2$ and so factors through $\cl(K)$;
however, this contradicts the hypothesis that $3 \nmid h(K)$.

This completes the proof of Theorem \ref{hadano}.

\section{Multiplicities of mod-2 dihedral cuspforms in weight 2}\label{sec:mult}

The following conjecture complements Theorem \ref{mf2}. 
Note that the fact that ${\mm \subset \TT_2(N)}$ is a maximal ideal automatically implies that $\dim S_2(N)_\mm \geq 1$.

\begin{conj}\label{mfmult}
Let $N$ be an odd prime and $\mm$ a maximal ideal of $\TT_2(N)$.
\begin{enumerate}
\item Suppose $N \equiv 1 \cmod{8}$.
\begin{enumerate}
\item\label{1pos} If $\mm$ is $\KNp$-dihedral, then $\dim S_2(N)_\mm \geq 4.$
\item\label{1neg} If $\mm$ is $\KNm$-dihedral, then $\dim S_2(N)_\mm \geq h(-N)^\even.$
\item\label{1red} If $\mm$ is reducible, then $\dim S_2(N)_\mm \geq \displaystyle \frac{h(-N)^\even - 2}{2}.$
\end{enumerate}
\item Suppose $N \equiv 5 \cmod{8}$.
\begin{enumerate}
\item\label{5pos} If $\mm$ is ordinary $\KNp$-dihedral, then $\dim S_2(N)_\mm \geq 4.$
\item\label{5neg} If $\mm$ is $\KNm$-dihedral, then $\dim S_2(N)_\mm \geq 2.$
\end{enumerate}
\item \label{37dih} Suppose $N \equiv 3 \cmod{4}$ and $K = \KK{\pm N}$.
\begin{enumerate}
\item\label{37diha} If $\mm$ is ordinary $K$-dihedral, then $\dim S_2(N)_\mm \geq 2$.
\end{enumerate}

\end{enumerate}
\end{conj}

In the case that $N \equiv 9 \cmod{16}$, part \eqref{1red} has been proved by Calegari and Emerton \cite[Theorem 1.1]{calegari-emerton-eisenstein}: indeed, they establish that $\dim S_2(N)_\mm = \frac{h(-N)^\even - 2}{2}$ for the unique reducible $\mm$ in this case. 

\begin{prop}\label{mfmultproof}
Part \eqref{37dih} of Conjecture~\ref{mfmult} is true when $K = \KNm$.
\end{prop}

\begin{proof}
If $K = \KNm$, and $N \equiv 3 \cmod{4}$ is a prime, and $\eps = \eps_K$, then there are exactly $\frac{h(-N) - 1}{2}$ distinct $K$-dihedral forms in $S_1(N, \eps, \CC)$ corresponding to inductions of characters $\psi: \Gal(H(K)/K) \to \CC^\times$ (see, for example, \cite[\S 8.1.I]{serre:weight1} for details). Since $h(-N)$ is odd, all of these reduce to distinct representations modulo $2$, so that $S_1(N, \eps_K, \bar\FF_2)^{K\myhyphen{\rm dih}}$ splits as a Hecke module into a direct sum of $\frac{h(-N) - 1}{2}$ non-isomorphic one-dimensional lines spanned by ordinary forms. The two maps $S_1(\Gamma_1(N), \FF_2) \into S_2(\Gamma_1(N), \FF_2)$ given by $f \mapsto f^2$ and $f \mapsto E_{1, \eps} f$ preserve Hecke eigenspaces (the former because we are in characteristic~$2$; the latter because $E_{1, \eps}$ in characteristic zero lifts the Hasse invariant: see user Electric Penguin's answer to \href{https://mathoverflow.net/a/228596/86179}{MathOverflow question 228497}\footnote{\url{https://mathoverflow.net/questions/228497}}) and are linearly independent \cite[Prop. 4.4]{edix2}. Since $\eps$ is quadratic, we obtain a Hecke equivariant embedding $\big(S_1(N, \eps, \bar\FF_2)^{K\myhyphen{\rm dih}} \big)^{2} \into S_2(N, \bar \FF_2)$ that doubles the eigenspace. 
\end{proof}

\section{Comparison with experimental results}
\label{conclusions}

To conclude, we compare our results to the empirical assertions about the mod-$2$ reduction of $T_2$ acting on $S_2(\Gamma_0(N), \QQ)$ for $N$ prime from the introduction.

\begin{itemize}
\item For $N \equiv 3 \cmod{8}$, the eigenvalue $0$ always occurs if $N > 3$.
\item For $N \equiv 1,3,5 \cmod{8}$, the eigenvalue 1 always occurs if $N > 163$.
\item For $N \equiv 1 \cmod{8}$, the eigenvalue $0$ occurs with probability $16.8\%$.
\item For $N \equiv 5 \cmod{8}$, the eigenvalue $0$ occurs with probability $42.2\%$.
\item For $N \equiv 7 \cmod{8}$, the eigenvalue $0$ occurs with probability $17.3\%$.
\item For $N \equiv 7 \cmod{8}$, the eigenvalue $1$ occurs with probability $47.9\%$.
\end{itemize}
Of these, the first assertion is implied by part (\ref{countss3}) of Theorem~\ref{mf2}
and the second assertion is implied by part (\ref{countord}) of Theorem~\ref{mf2}. Combining the other parts
of Theorem~\ref{mf2} with the Cohen-Lenstra heuristics yields the following statements.
\begin{itemize}
\item For $N \equiv 5 \cmod{8}$, the eigenvalue $0$ occurs for ``dihedral reasons'' when $u(N) \equiv 1 \cmod{2\OO(N)}$. The three possible nonzero reductions of $u(N)$ mod $2\OO(N)$ being equally likely, this should occur with probability $\frac{1}{3} = 33.3\%$.

\item For $N \equiv 7 \cmod{8}$, the eigenvalue $1$ occurs for ``dihedral reasons'' when
$h(N) > 1$ or $h(-N)^{\odd, \twosplit} > 1$.
Each of these is modeled by the probability that a random finite abelian group, modulo the subgroup generated by a random element, yields a nontrivial quotient; this probability is
\[
1 - \prod_{p>2} \prod_{j=1}^\infty \left( 1 - \frac{1}{p^{j+1}} \right) = 0.2455\dots.
\]
Since the two events are presumed to be independent, at least one should occur with probability $43.1\%$.
\end{itemize}

Removing these cases leaves the following occurrence of eigenvalues arising from exceptional or big-image maximal ideals.
\begin{itemize}
\item For $N \equiv 1 \cmod{8}$, the eigenvalue $0$ occurs with probability $16.8\%$.
\item For $N \equiv 5 \cmod{8}$, the eigenvalue $0$ occurs with probability $13.3\%$.
\item For $N \equiv 7 \cmod{8}$, the eigenvalue $0$ occurs with probability $17.3\%$.
\item For $N \equiv 7 \cmod{8}$, the eigenvalue $1$ occurs with probability $8.4\%$.
\end{itemize}
It would of course be desirable to explain these probabilities also. This will require combining some analysis of the corresponding representations with Wood's nonabelian analogue of the Cohen-Lenstra heuristics \cite{wood}, which for a given pair of finite groups $G,G'$ predicts the probability that a quadratic number field $K$ admits a Galois $G$-extension $L$ for which $L/\QQ$ is a Galois $G'$-extension. 

For $N < 200000$ prime, we also checked whether Theorem~\ref{mf2} and Conjecture~\ref{mfmult} 
together give a sharp lower bound on the eigenvalue multiplicities of 0 and 1. For each residue mod 8, the percentage of cases where this fails
is shown in Table~\ref{multiplicity results}.

\begin{table}[ht]
\caption{Frequency of unexplained eigenvalue multiplicity in the mod-2 reduction of $T_2$ on $S_2(\Gamma_0(N), \QQ)$ for $N < 200000$ prime.}
\begin{tabular}{c|c|c}
$N$ mod $8$ & excess multiplicity of 0 & excess multiplicity of 1 \\
\hline
1 & 16.4\% & 43.8\% \\
3 & 53.0\% & 45.7\% \\
5 & 22.5\% & 45.8\% \\
7 & 17.3\% & 39.0\% \\
\end{tabular}
\label{multiplicity results}
\end{table}
Note that these percentages include both uncounted (exceptional or big-image) maximal ideals and non-sharpness in Conjecture~\ref{mfmult}. The preceding calculation suggests that excess multiplicity of 0 for $N \equiv 1,7 \cmod{8}$ arises almost entirely from uncounted maximal ideals, but in other cases Conjecture~\ref{mfmult} may need to be refined.


\begin{thebibliography}{99}

\bibitem{bard}
G. Bard, \textit{Algebraic Cryptanalysis}, Springer, Dordrecht, 2009.

\bibitem{bennett-rechnitzer}
M. Bennett and A. Rechnitzer, Computing elliptic curves over $\mathbb{Q}$: bad reduction at one prime,
in \textit{Recent progress and modern challenges in applied mathematics, modeling and computational science},
Fields Inst. Commun., 79, Springer, New York, 2017, 387--415.

\bibitem{birch}
B. Birch, Hecke actions on classes of ternary quadratic forms,
\textit{Computational number theory (Debrecen, 1989), de Gruyter, Berlin, 1991}, 191--212.

\bibitem{bcdt}
C. Breuil, B. Conrad, F. Diamond, and R. Taylor, On the modularity of elliptic curves over $\mathbb{Q}$:
wild 3-adic exercises, \textit{J. Amer. Math. Soc.} \textbf{14} (2001), 843--939.

\bibitem{brumer}
A. Brumer, The average rank of elliptic curves. I,
\textit{Invent. Math.} \textbf{109} (1992), 445--472.

\bibitem{calegari-emerton-eisenstein}
F. Calegari and M. Emerton, On the ramification of Hecke algebras at Eisenstein primes, \textit{Invent. Math.} \textbf{160} (2005), 97--144.

\bibitem{calegari-emerton-moddeg}
F. Calegari and M. Emerton, Elliptic curves of odd modular degree,
\textit{Israel J. Math.} \textbf{169} (2009), 417--444.

\bibitem{cohen-lenstra}
H. Cohen and H.W. Lenstra, Jr., Heuristics on class groups,
\textit{Number theory (New York, 1982)},
Lecture Notes in Math.\ 1052, Springer, Berlin, 1984, 26--36.

\bibitem{cremona}
J.E. Cremona, \textit{Algorithms for Modular Elliptic Curves},
second edition, Cambridge Univ. Press, 1997;
\url{http://homepages.warwick.ac.uk/~masgaj/book/fulltext/index.html}.

\bibitem{dieulefait}
L. Dieulefait, Remarks on Serre's modularity conjecture,
\textit{Manuscr. Math.} \textbf{139} (2012), 71--89.

\bibitem{edix}
Bas Edixhoven, The weight in Serre's conjecture on modular forms, \textit{Invent. Math.} \textbf{109} (1992), 563--594.

\bibitem{edix2}
B. Edixhoven, Comparison of integral structures on spaces of modular forms of
weight two, and computation of spaces of forms mod 2 of weight one,
\textit{J. Inst. Math. Jussieu} \textbf{5} (2006), 1--34.

\bibitem{lmfdb}
The LMFDB Collaboration, \textit{L-Functions and Modular Forms Database},
\url{http://www.lmfdb.org} (retrieved February 2018).

\bibitem{hadano}
T. Hadano, On the conductor of an elliptic curve with a rational point of order 2, \textit{Nagoya Math. J.}
\textbf{53} (1974), 199--210.

\bibitem{hein}
J. Hein, Orthogonal modular forms: an application to a conjecture of Birch,
algorithms and computations, PhD thesis, Dartmouth College, 2016.

\bibitem{hein-github}
J. Hein, \texttt{git} repository \url{https://github.com/jefferyphein/ternary-birch} (retrieved February 2018).

\bibitem{hortsch}
R. Hortsch, Counting elliptic curves of bounded Faltings height,
\textit{Acta Arith.} \textbf{173} (2016), 239--253.

\bibitem{kw1}
C. Khare and J.-P. Wintenberger, Serre's modularity conjecture (I), \textit{Invent. Math.} \textbf{178} (2009), 485--504.

\bibitem{kw2}
C. Khare and J.-P. Wintenberger, Serre's modularity conjecture (II), \textit{Invent. Math.} \textbf{178} (2009), 505--586.

\bibitem{kida}
M. Kida, Ramification in the division fields of an elliptic curve, \textit{Abh. Math. Sem. Univ. Hamburg}
\textbf{73} (2003), 195--207.

\bibitem{kilford-wiese}
L.J.P. Kilford and G. Wiese, On the failure of the Gorenstein property for Hecke algebras of prime
weight, \textit{Exper. Math.} \textbf{17} (2008), 37--52.

\bibitem{kisin}
M. Kisin, Modularity of 2-adic Barsotti-Tate representations, \textit{Invent. Math.} \textbf{178} (2009), 587--634.

\bibitem{mazur}
B. Mazur, Modular curves and the Eisenstein ideal, \textit{Publ. Math. IH\'ES} \textbf{47} (1977), 33--186.

\bibitem{mazurserre}
B. Mazur and J.-P. Serre, Points rationnels des courbes modulaires $X_0(N)$ (d'apr\`es A. Ogg),
\textit{S\'eminaire Bourbaki (1974/1975), Exp. No. 469}, Lecture Notes in Math. 514, Springer, Berlin, 1976, 238--255.

\bibitem{mestre}
J.-F. Mestre, La m\'ethode des graphes. Exemples et applications,
in \textit{Proceedings of the international conference on class numbers and fundamental units of algebraic number fields (Katata, 1986)}, Nagoya Univ., Nagoya, 1986, 217--242.

\bibitem{ogg}
A. Ogg, Abelian curves of small conductor, \textit{J. reine angew Math.} \textbf{226} (1967), 204--215.

\bibitem{pari}
The PARI group, \texttt{PARI/GP}, \url{http://pari.math.u-bordeaux.fr/}, version 2.9.4 (2017).

\bibitem{randriam}
H. Randriam, Hecke operators with odd determinant and binary frameproof codes beyond the probabilistic bound?,
\textit{2010 IEEE Information Theory Workshop}, Dublin, 2010, 1--5.

\bibitem{roberts}
D. Roberts, Newforms with rational coefficients, \textit{Ramanujan J.} (2018), 28 pages.

\bibitem{sage}
The Sage Development Team, \texttt{Sage}, \url{http://www.sagemath.org}, version 8.1 (2017).

\bibitem{serre:conj}
J.-P. Serre, Sur les repr\'esentations modulaires de degr\'e 2 de $\Gal(\QQbar/\QQ)$,
\textit{Duke Math. J} \textbf{54} (1987), 179--230.

\bibitem{serre:weight1}
J.-P. Serre, Modular forms of weight one and Galois representations, {\it Algebraic Number Fields} (A. Fr\"ohlich, ed.), Academic Press (1977), pp.193--268.

\bibitem{serre:ecs}
J.-P. Serre, Propriet\'es galoisiennes des points d'ordre fini des courbes elliptiques,
\textit{Invent. Math.} \textbf{15} (1972), 259--331.

\bibitem{setzer}
B. Setzer, Elliptic curves of prime conductor, \textit{J. London Math. Soc.} \textbf{10} (1975), 367--378.

\bibitem{stein-watkins}
W.A. Stein and M. Watkins,
A database of elliptic curves--first report,
\textit{Algorithmic number theory (Sydney, 2002)}, Lecture Notes in Comput.\ Sci.\ 2369, Springer, Berlin, 2002, 267--275.

\bibitem{taguchi}
Yuichiro Taguchi, Induction formula for the Artin conductors of mod $\ell$ Galois representations,
\textit{Proc. Amer. Math. Soc.} \textbf{130} (2002), 2865--2869.

\bibitem{magma}
The University of Sydney Computational Algebra Group, \texttt{Magma},
\url{http://magma.maths.usyd.edu.au/magma/}, version 2.23-8 (2018).

\bibitem{voight-agct}
J. Voight, Computing classical modular forms as orthogonal modular forms,
lecture notes (2017) available at \url{http://www.cirm-math.fr/ProgWeebly/Renc1608/voight.pdf}.

\bibitem{wood}
M.M. Wood, Nonabelian Cohen-Lenstra moments, arxiv:1702.04644v1 (2017).

\bibitem{yoo}
H. Yoo, On Eisenstein ideals and the cuspidal group of $J_0(N)$, \textit{Israel J. Math.}
\textbf{214} (2016), 359--377.

\end{thebibliography}
\end{document}